\documentclass[12pt]{article}

\usepackage{amsmath,amssymb,amsthm,amsfonts}

\newtheorem{thm}{Theorem}[section]

\newtheorem{lem}[thm]{Lemma}
\newtheorem{prop}[thm]{Proposition}
\newtheorem{rmk}[thm]{Remark}

\newcommand{\pd}{\partial}
\newcommand{\ke}{K\"ahler-Einstein metric}
\newcommand{\K}{K\"ahler }
\newcommand{\gt}{Grauert tube}
\newcommand{\ma}{Monge-Amp\`ere }

\newcommand{\spc}{strictly pseudoconvex }

\begin{document}
\baselineskip=16pt
\title{Complete Ricci-flat metrics through  a rescaled exhaustion } 
\author{Su-Jen Kan\thanks {This research is partially supported by NSC97-2628-M-001-023-MY3 of
Taiwan. {\it 2000 Mathematics Subject Classification.} 32Q20, 32Q25 }}

\date{Sep. 20, 2010}

\maketitle
\begin{abstract}
Typical existence  result on  Ricci-flat metrics is in manifolds  of
finite geometry, that is, on
$F=\overline F-D$ where
$\overline F$ is a compact K\"ahler manifold and  $D$ is a smooth divisor. 

We view this existence problem from a different perspective.  For a given complex manifold $X$, we
take a suitable exhaustion $\{X_r\}_{r>0}$  admitting  complete \ke s  of negative Ricci. Taking a
positive decreasing sequence $\{\lambda_r\}_{r>0},
\lim_{r\to\infty}\lambda_r=0$, we rescale the metric so that $g_r$ is the complete \ke\  in
$X_r$ of Ricci curvature $-\lambda_r$. The idea is to show the limiting metric
$\lim_{r\to\infty} g_r$ does exist. If so, it  is a  Ricci-flat metric in $X$.
Several examples: $X=\mathbb C^n$ and $X=TM$ where $M$ is a compact rank-one symmetric space have
been studied in this article.

The existence of complete \ke s of negative Ricci in 
bounded domains of holomorphy is well-known. Nevertheless, there is very few known for unbounded
cases. In the last section we show the existence, through exhaustion, of such kind of metric in the
unbounded  domain $T^{\pi}H^n$.
\end{abstract}

\setcounter{equation}{0}
\setcounter{section}{-1}
\section{Introduction} 

The goal of this paper is to looking for a way to construct a complete Ricci-flat metric.

By Yau's solving to the Calabi conjecture, complete Ricci-flat metrics have existed in 
any  compact K\"ahler 
manifold with vanishing
first Chern class. There is no general existence theorem for the non-compact case yet.
 The most general existence is
due to Tian-Yau \cite {T-Y} on manifolds $F$ of
finite geometry, i.e.,
$F=\overline F-D$ where
$\overline F$ is a compact K\"ahler manifold and  $D$ is a smooth divisor.  
 Perturbing a
suitable chosen \ke\ near the divisor
 followed by  the continuity method, they are able to 
conclude the existence of a Ricci-flat metric. Their convergence argument has heavily relied on
 Sobolev inequalities.  For the choice of a  background metric near $D$ and the use of Sobolev
inequalities, certain topological condition on $M$ has been asked for. The general requirement is 
 $0<c_1(\overline
F)=\alpha c_1(L_D),\alpha\ge 1.$

We will deal with this problem from a  different  point of view  motivated by the following
example.  The complete \ke\ of Ricci 
$-(n+1)$ for the ball
$B_r(0)\subset \mathbb C^n$ is given as  $g_r=\sum_{ j}\frac 1{r^2-|z|^2}dz_jd\bar z_j-\sum_{i,
j}\frac {\bar z_i z_j}{(r^2-|z|^2)^2}dz_id\bar z_j.$ The rescaled metric $r^2g_r$  has Ricci
curvature
$-\frac {(n+1)}{r^2}$. As $r\to\infty$, the limit metric is then a  K\"ahler-Einstein with vanishing
Ricci curvature in
$\mathbb C^n$ if it existed. Indeed, $\lim_{r\to \infty}r^2g_r=\sum_{ j}dz_jd\bar z_j.$

The second example is the tangent bundle $TM$ of a compact symmetric space of rank-one.
 $TM$ is exhausted by disk bundles $\{T^rM\}_{r>0}$, named as \gt s when $TM$ is
equipped with the adapted  complex structure.
As a bounded smooth \spc domain in the \K
manifold $TM$, the existence of a complete \ke \ with negative Ricci  in $T^rM$ is
guaranteed. Furthermore, potential functions  could be represented by  ordinary differential
equations as discussed in
\cite {K2}.
For some suitably chosen decreasing positive numbers $\{\lambda_r\}_{r>0},\; \lim_{r\to\infty}
\lambda_r=0$, we pick the  complete \ke\  $g_r$ with Ricci  $-\lambda_r$ in $T^rM$ and let
$K_r$ be its \K potential uniquely determined by the corresponding ODE.
Through some analysis in the ODEs, the family $\{K_r\}_{r>0}$ has a $C^2$-convergence as
$r\to\infty$. The limiting function is then a \K potential of a Ricci-flat metric in $TM$. It is also
interesting to observe that exhausting $\mathbb C^n$ through unbounded domains $T^r\mathbb R^n$ has
achieved the same Ricci-flat as exhausting $\mathbb C^n$ through balls.

On the other hand, Stenzel \cite {S} has worked on  the same $TM$ as well. By the transitivity of the
rank-one symmetry,  the defining
 equation for a Ricci-flat metric could be reduced to an ordinary differential equation. Working
directly on the solvability and the completeness of this  ordinary differential equation,
Stenzel was able
to show the existence  a complete Ricci-flat metric in $TM$ for compact rank-one
$M$.

Although our resulting Ricci-flat metric  through the rescaling process turns out to coincide
with the one constructed by Stenzel in $TM$.  We think  the
approach here is interesting and we are looking for some further investigation.

 Cheng-Yau  have  proved
the existence of a  complete K\"ahler-Einstein metric of negative Ricci  in 
bounded weakly
pseudoconvex domains in
$\Bbb C^n$ through the exhaustion by  a family of bounded smooth strictly pseudoconvex subdomains.
Later on  Mok-Yau
have generalized the existence to any bounded Stein domain and use the existence as a
characterization of a bounded domain of holomorphy.
The convergence of the exhaustion has strongly
relied on the boundedness of the domain $\Omega$. An essential point is $\Omega\Subset B(0,R)$ for
some large $R$, so that the Poincar\'e metric of  $B(0,R)$ could be used as a comparison to get hold a
uniform lower bound of the exhaustion.

It is not clear whether such kind of metric exists in unbounded Stein domain or not. The
difficulty  is on the lower bound. In the last part of this
article, we study the unbounded domain $T^{\pi}H^n$. Working on the ODEs, 
we show a uniform bound needed in the
convergence argument could be obtained. We conclude  
there exists a complete \ke\ with negative Ricci curvature in
$T^{\pi}H^n$.

We started from the ball example, a detailed convergence argument has been
provided in the first section. In $\S2$,
 fundamental properties on the \gt s' setting and related ODEs are established. In
$\S3$, we take care of the rescaling process. By some suitable choice of the rescaling factors on
the ODEs, the convergence could be achieved. 
$\S4$ is on the existence of a complete \ke\ with negative Ricci curvature in  $T^{\pi}H^n$. The 
key point
is a lower bound estimate shown on Lemma 4.1. Some holomorphic sectional curvatures are also 
computed in Proposition 4.5.

The first   version of this preprint has been completed on January 2009.
Some of the main idea and the essential part of this paper has been initiated and done when the
author was visiting the Ruhr-Universit\"at Bochum, Germany in the Fall of 2007. I would like to thank
the complex geometry group there for the generous support and the hospitality during my visit. 
I would also like  to thank Professor Ryoichi Kobayashi for bringing my attention to this problem and Professor Damin Wu for pointing out a mistake in the first draft.

\section{A  Ricci-flat obtained from a rescaled exhaustion.}
\setcounter{equation}{0}
\setcounter{thm}{0}

Let $B_r=\{(z_1,\cdots,z_n)\in \Bbb C^n: |z_1|^2+\cdots +|z_n|^2<r^2\}$ be the ball of radius $r$ in 
$\Bbb C^n$ which has admitted a  complete  \ke\ of Ricci curvature $-(n+1):$ 
\begin{equation}\begin{aligned}
g_r(z)&=-\pd\bar\pd\log (r^2-|z|^2)\\
&=\sum_{ j}\frac 1{r^2-|z|^2}dz_jd\bar z_j-\sum_{i, j}\frac {\bar z_i z_j}{(r^2-|z|^2)^2}dz_id
\bar z_j.
\end{aligned}\end{equation}
 The metric $r^2g_r$ is then a complete \ke\ of Ricci  
$\frac {-(n+1)}{r^2}$ in  $B_r$ and 
$$\lim_{r\to \infty}r^2g_r(z)=\sum_{ j}dz_j\bar z_j$$ has created  a complete Ricci-flat metric 
in the exhaustion space $\Bbb C^n$.

 This is the
example  motivating this work of rescaling  \ke s\  to
achieve a Ricci-flat metric.

The exhaustion process could also be checked from the potential level. By \cite {C-Y}, there is a
unique real-analytic  function $w^r$ in $B_r$ such that 
\begin{equation*}\left\{\begin{aligned}
\det (w^r_{i\bar j})&= e^{(n+1)w^r},\\
(w^r_{i\bar j})>0\;\text{ in } B_r,&\;\;\; w^r=\infty \text{ on } \pd B_r
\end{aligned}\tag {1.2}\right.\end{equation*}
and $\sum w^r_{i\bar j}dz_idz_{\bar j}$ gives the unique complete \ke\  of Ricci curvature
$-(n+1).$ For $h:=-\log (-\varphi)$, the authors have derived the following formula in \cite {C-Y}
\begin{equation*}\det (h_{i\bar j})=\left(\frac{-1}{\varphi}\right )^{n+1}\det (\varphi_{i\bar
j})(-\varphi+|d\varphi|^2).
\end{equation*}
Using this equation, the solution $w^r$ of (1.2) can be computed explicitly 
\begin{equation*}
w^r(z)=-\log (r^2-|z|^2)+\frac 2{n+1}\log r.
\tag {1.3}\end{equation*}
It is clear that $w^r(z)$ is an increasing function of $|z|$ and its minimum has occurred at the
origin:  \begin{align*}
a_r:=\inf_{B_r} w^r=w^r(0)=\frac {-2n}{n+1}\log r\to-\infty\text{  as  } r\to\infty.
\tag {1.4}\end{align*}
For a fixed $z\in  B_r$,
$$\frac {\pd}{\pd r} w^r(z)=\frac {-2(nr^2+|z|^2)}{(n+1)r(r^2-|z|^2)}<0.$$
Hence   $w^r(z)<w^s(z),\forall z\in B_s\Subset B_r.$ 

Due to (1.4), the family 
 $\{w^r\}_{r>0}$ will diverge as $r$ goes to infinity. 
Searching for some convergent \K potentials, we consider the following $\{\tilde w^r\}$ instead.
Let
$\tilde w^r$ be the unique
solution of 
\begin{equation*}\left\{\begin{aligned}
\det (\tilde w^r_{i\bar j})&= e^{\frac{(n+1)}{r^2}\tilde w^r},\; \\
(\tilde w^r_{i\bar j})>0\;\text{ in } B_r,&\;\;\; \tilde w^r=\infty \text{ on } \pd B_r.
\end{aligned}\tag {1.5}\right.\end{equation*} 
For any $r>0$, $\tilde w^r$ is then a \K potential of  the complete \ke\ of Ricci
$\frac{-(n+1)}{r^2}$ in $B_r$. A comparison between different $\tilde w^r$ could be derived.

\begin{lem} 
 $$\aligned\tilde w^r(z)&=r^2\log\frac {r^2}{r^2-|z|^2}\ge 0;\\
\tilde w^r(z)&<\tilde
w^s(z), \forall z\in B_s\Subset B_r.
\endaligned$$  
\end{lem}
\begin{proof}
Let $u^r=w^r-a_r,$ then
\begin{align*}
\det (u^r_{i\bar j})= e^{(n+1)u^r} e^{(n+1)a_r}.
\end{align*}
 \begin{align*}\det \left ( (e^{\frac{n+1}{-n}a_r}u^r)_{i\bar j}\right )&= e^{(n+1)u^r}\\
&=\exp\{\frac {(n+1)}{r^2} (e^{\frac{n+1}{-n}a_r}u^r)\},
\end{align*} $\left ( e^{\frac{n+1}{-n}a_r}u^r_{i\bar j}\right )>0$ in $B_r$ and
$e^{\frac{n+1}{-n}a_r}u^r=\infty$ on $\pd B_r$. By the uniqueness of the solution in (1.5),
\begin{align*} \tilde w^r(z)&= e^{\frac{n+1}{-n}a_r}u^r(z)\\
&=r^{2}(w^r(z)-a_r)\\
&= r^2(-\log(r^2-|z|^2)+\log r^2)\ge 0,\; z\in B_r,\; \forall r> 0.
\tag {1.6}\end{align*}
 Taking derivative of $\tilde w^r$ with respect to $r$,
\begin{align*}H^r(z):=\frac {\pd}{\pd r} \tilde w^r(z)=2r\log\frac {r^2}{r^2-|z|^2}-\frac
{2r|z|^2}{r^2-|z|^2}.\tag {1.7}\end{align*}
 Notice that
$H^r$ is actually a real-valued function of $|z|^2$ and $H^r(0)=0$. 
Viewing $|z|^2=a\in [0,r^2)$ and taking $H^r$ as a function of $a$, we now  take 
derivative with respect to
$a$,
\begin{align*}
\frac {\pd}{\pd a}H^r(a)=\frac {-2ra}{(r^2-a)^2}=\frac {-2r|z|^2}{(r^2-|z|^2)^2}\le 0.
\tag {1.8}\end{align*}
Thus for any $z\in B_r$, we have $$H^r(z)\le H^r(0)=0,\; \forall z\in B_r,\;\forall r>0.$$
This together with (1.7) has shown that for $z$ fixed $\tilde w^r(z)$ is decreasing with respect to
$r$. The comparison 
$\tilde w^r(z)<\tilde w^s(z), \forall z\in B_s\Subset B_r$ is achieved.
\end{proof}
 
\begin{thm} 
The family  $\{\tilde w^r\}$  converges $C^2$-smoothly to a $C^2$ function in $\mathbb C^n$. 
$$\lim_{r\to\infty}\tilde w^r(z)=|z|^2$$
\end{thm}
\begin{proof}
By Lemma
1.1, the limit has existed because  $\{\tilde w^r\}$ are uniformly bounded on any compact subset.   
Viewing
$r^2=\mu$ and  using the  L'H\^opital's rule repeatedly, the 
limit can be computed explicitly
\begin{align*}
\lim_{r\to\infty}\tilde w^r(z)&=\lim_{\mu\to\infty}\frac {\log {\frac {\mu-|z|^2}{\mu}}}{-\frac
1\mu}\\ &=\lim_{\mu\to\infty}\frac {\mu |z|^2}{\mu -|z|^2}\\
&=|z|^2.
\end{align*}
The first and 2nd order convergence can also be computed directly.
$$\lim_{r\to\infty}\tilde w_i^r(z)=\lim_{r\to\infty}\frac {r^2\bar z_i}{r^2-|z|^2}=\bar z_i;$$
$$\tilde w_{i\bar j}^r(z)=\frac {\delta_{ij} r^2
(r^2-|z|^2)+r^2\bar z_i z_j}{(r^2-|z|^2)^2}.$$ The resulting  metric 
$\lim_{r\to\infty}\tilde w_{i\bar j}^r(z)=\delta_{ij}$ is exactly the Euclidean metric.
\end{proof}

\section{Properties of K\"ahler-Einstein potentials
in tangent bundles of rank-one symmetric spaces}

\setcounter{equation}{0}
\setcounter{thm}{0}

Terminologies used in this section come from \cite {K2} and references listed there, we will explain
very briefly what a \gt\ is. 

For any  real-analytic Riemannian manifold $M$, there exists a  neighborhood
$U$ of $M$ in $TM$ where the adapted complex structure can be endowed with. 
The   adapted complex structure is the
unique complex structure
 turning every leaf of the Riemannian
foliation into a holomorphic curve. With respect to this complex structure, the length square
function  is  real-analytic and 
 strictly plurisubharmonic. 

A {\it Grauert tube of radius $r$ over center $M$} is the disk bundle 
$$T^rM\ = \{(x,v)\in TM: x\in M, |v|<\frac {r}2\}$$
equipped with the adapted complex structure. Each $M$ has associated with a maximal possible
radius, denoted by $r_{max}(M)$, such that the adapted complex structure could be defined in
$T^{r_{max}(M)}M$. In this article, $\rho (x,v)=4|v|^2.$

The maximal radius $r_{max}(M)=\infty$ when
 $M$ is a compact symmetric space of rank-one, and for any $0<r<\infty$ the \gt\ $T^rM$ is a Stein
manifold with bounded \spc boundary. The existence of a complete \ke\ with negative Ricci curvature
is guaranteed by \cite {C-Y}. Furthermore, as shown in  \cite{S}, \cite{A} and \cite{K2},
this \ke\ has a \K potential solely depending on the length square.

Let $\mathcal{S}_{M}$ denote the density function of a Riemannian symmetric space $M$ of rank-one. 
 It was shown in $\S5$ \cite{K2} that for any $\lambda>0$ and any $r<r_{max}(M)$,
there exists a \K potential $h_r(\sqrt\rho)$ for the complete \ke\ of  Ricci curvature $-\lambda<0$
in the  \gt\ $T^rM$. Let  $u=\sqrt\rho$,  the  potential 
function is unique and satisfies the following, {\it c.f.} (5.1) \cite {K2},
\begin{equation}
	h_r''(u)(h_r'(u))^{n-1}\exp (-\lambda h_r(u))= u^{n-1}\hat{\mathcal{S}}_{M}(u);\; \lim_{u\to
r}h_r(u)=\infty. 
\end{equation}
 The derivatives are taken with respect to $u$ where 
$\hat{\mathcal{S}}_{M}(u):=\mathcal{S}_{M}(-u^2)$.    
$$\hat{\mathcal{S}}(u)=2^{n-1}u^{1-n}\left (cosh\ \frac {u}2\right )^k (sinh\ \frac
{u}2)^{n-1},$$ 
where $k=n-1$ for the round sphere and the real projective space and $k=1,3,7$ for the complex
projective space, the quaternionic projective space    and  the Caley plane, respectively. 

We may
also consider two non-compact symmetric spaces: the real hyperbolic space $H^n$ and the  Euclidean
space $\mathbb R^n$. For  $H^n$: $\hat{\mathcal{S}}(u)=\left (\frac {\sin u}{u}\right )^{n-1}
,r_{max}(H^n)=\pi$. For  $\mathbb R^n$: $\hat{\mathcal{S}}(u)=1\;, r_{max}(\mathbb R^n)=\infty$. 

In $T^rM$, (2.1) can be expressed as 
\begin{equation}
\left\{\begin{aligned}
	h_r''(u)(h_r'(u))^{n-1}&=e^{ \lambda h_r(u)}\mathcal{D}_{M}(u),\; u\in [0,r)\\
\lim_{u\to r}h_r(u)&=\infty 
 \end{aligned}\right.
\end{equation}
where
$\mathcal{D}(u)=u^{n-1}$ for $\Bbb R^n$; $\mathcal{D}(u)=(\sin(u))^{n-1}$ for $H^n$;
$\mathcal{D}(u)=(\sinh (u))^{n-1}$ for the round sphere;
$\mathcal{D}(u)=2^{n-1}\left (cosh\ \frac {u}2\right )^k (sinh\ \frac {u}2)^{n-1}, k=1,3,7$ for
the complex projective space, the quaternionic projective space    and  the Caley plane,
respectively.
For all the above cases  $\mathcal D_{ M}(u)\ge 0$ and  $\mathcal D_{ M}(u)= 0$
if and only if  $u=0$. 

Some well-known properties of $h_r$ have been discussed in  \cite {K2}, we
summarize some of them here for  future application. $h_r(\rho)$ is a real-analytic
function of $\rho$, near the center, it has the  asymptotic expression 
$h_r(u)=a+bu^2+cu^4+O (6),\; b> 0.$

\begin{prop} 
\
\begin{enumerate}
\item[(1)]  $h_r'(u)>0,h_r''(u)>0, \;\forall u\in (0,r); \;\inf_{u\in [0,r)}h_r(u)=h_r(0);$
\item[(2)]  $h_r'(0)=0;$

\item[(3)]  $h_r'(t)=\left(\int_0^{t} ne^{ \lambda h_r(u)}\mathcal{D} (u)\right)^{\frac
1n},\;\forall t\in(0,r)$.

\item [(4)]  $\lim_{u\to r}h'_r(u)=\infty,\; \lim_{u\to
r}h''_r(u)=\infty$
\end{enumerate}
\end{prop}
\begin{proof} Since the right hand side of (2.2) is positive for all $u\in (0,r),$ neither $h_r''(u)$
nor
$h'_r(u)$ has any zero point in $u\in (0,r)$.
That is,  $h'_r(u)$ is either positive or negative
for all $u\in (0,r)$. Since $\lim_{u\to r}h_r(u)=\infty$, the only possibility is $h'_r(u)>
0,\forall u\in (0,r).$ The positivity of $h_r''(u)$ 
follows. 

It was shown in Prop.5.1 of \cite{K2} that  $h_r$ is a real-analytic function of
$\rho=u^2$, hence $h_r'(0)=\frac {dh_r}{du}(0)=0.$ By (2.2)

\begin{equation}
\begin{aligned}
\frac {d}{du}(h_r'(u))^n&=ne^{ \lambda h_r(u)}\mathcal{D} (u),\\ 
(h_r'(t))^n&=\int_0^{t} ne^{ \lambda h_r(u)}\mathcal{D} (u).
 \end{aligned}
\end{equation}
Through the relation $h_r(u)=\int_{0}^{u} h'_r(t)dt
+h_r(0)$, the condition $\lim_{u\to r}h_r(u)=\infty$ has implied  that 
\begin{equation}\lim_{u\to r}h'_r(u)=\infty \text{ and similarly, } \lim_{u\to
r}h''_r(u)=\infty.\end{equation} \end{proof}
We also need some estimate for $h''_r(0)$.
\begin{lem} 
$h''_r(0)=e^{\frac {\lambda}{n}h_r(0)}> 0.$
\end{lem}
\begin{proof}
Let \begin{equation} F_r(u)= {h'_r(u)}{\mathcal{D}^{\frac {-1}{n-1}}_{M}(u)}.\end{equation}
The equation (2.2) could be written as 
\begin{equation}
h_r''(u)(F_r(u))^{n-1}=e^{ \lambda h_r(u)},\; u\in [0,r).
\end{equation}
 From (2.5)
\begin{equation}
h_r''(u)=F'_r(u){\mathcal{D}^{\frac {1}{n-1}}_{M}(u)}+F_r(u)\left ({\mathcal{D}^{\frac
{1}{n-1}}_{M}}\right )'(u)
\end{equation}
For $\mathbb R^n$, the real hyperbolic space $H^n$ and all the compact rank-one symmetric spaces,
${\mathcal{D}^{\frac {1}{n-1}}(0)}=0$ and  $\left ({\mathcal{D}^{\frac {1}{n-1}}}\right )'(0)=1.$
Hence
\begin{equation}
h_r''(0)=F_r(0)
\end{equation}
Plugging (2.8) into (2.6) with $u=0$, the lemma is concluded.\end{proof}
Let $K$ be a compact subset of $T^sM\Subset T^rM$. We would like to develop a comparison
between
$h_r(u)$ and $h_s(u)$ for $u\in K$. 
\begin{lem} $s<r$. 
Let $h_r,h_s$ be solutions of (2.2) in $T^rM$ and in $T^sM$, respectively. Then
$h_r(u)\le h_s(u),\forall u\in [0,s).$
\end{lem}
\begin{proof}  
 $F(u):=h_s(u)-h_r(u)$ is a continuous function defined in $[0,s)$, thus
$F^{-1}(-\infty,0]$ is a closed subset of $[0,s)$ which could be written as  union of closed
connected intervals  in $[0,s)$. There are two
cases.

{\bf Case 1.}
If $0\notin F^{-1}(-\infty,0]$, then  $h_r(u)< h_s(u), \forall u\in
[0,s)$.
\
 
Since  $h_r(0)< h_s(0)$ and $\lim_{u\to s}h_s(u)=\infty$,
either 
$h_r< h_s$ in the whole interval $[0,s)$ or there exists an $\alpha\in (0,s)$ such that
$h_r(\alpha)= h_s(\alpha)$ and $h_r(u)< h_s(u)$ for any $u\in [0,\alpha).$ 
If such  $\alpha$ exists,
\begin{equation}\begin{aligned}
 (h_r'(\alpha))^n&=\int_0^{\alpha} ne^{ \lambda h_r(u)}\mathcal D (u)\\
&< \int_0^{\alpha} ne^{ \lambda h_s(u)}\mathcal D (u)\\
&=(h_s'(\alpha))^n
\end{aligned}\end{equation}
which shows $h'_r(\alpha)< h'_s(\alpha)$. Thus there exists some $\epsilon>0$ such that  $h_r(u)<
h_s(u)$ for any
$u\in (\alpha,\alpha+\epsilon)$. The point $\alpha$ is then a local minimum of
the function $F$, then $0=F'(\alpha)=h'_s(\alpha)-h'_r(\alpha)$, a contradiction. Therefore, $h_r<
h_s$ in the whole interval $[0,s)$.

{\bf Case 2.} If $0\in F^{-1}(-\infty,0],$ then $0$ is an isolated point and  $h_r(u)\le h_s(u),
\forall u\in [0,s)$.

Let $I$ be the interval containing $0$, we claim  $I=\{0\}.$
Since $I$ is a closed subset of $[0,s)$ containing $0$, it is  either $[0,s)$
 or there exists a 
$\delta\ge 0$ such that
$I=[0,\delta]$ is a maximal connected interval in $F^{-1}(-\infty,0]$.

The first case
has been ruled out because $\lim_{u\to s}h_s(u)=\infty$. The second case means 
$h_s(u)\le h_r(u)$ in $[0,\delta],\; h_s(\delta)= h_r(\delta) $ and $h_s(\delta+\epsilon) >
h_r(\delta+\epsilon)$ for 
$0<\epsilon\ll 1$. Similar calculation as (2.9) has led to  
$h_r'(\delta)\ge h_s'(\delta)$ if $\delta>0$. This shows  for $0<\epsilon\ll 1$, 
$h_r(\delta+\epsilon)\ge h_s(\delta+\epsilon)$, a contradiction. Thus
$\delta=0$ and $I=\{0\}.$ 

Let $(0,\beta)$ be a maximal open interval  such that $h_r< h_s$ in
$(0,\beta).$ If $\beta<s,$ then $h_r(\beta)= h_s(\beta)$ and $h_r> h_s$ in $(\beta,\beta+\epsilon')$
for some
$\epsilon'\ll 1$.
Following (2.9),
$h'_r(\beta)< h'_s(\beta)$. Since $h_r(\beta)= h_s(\beta)$, it is not possible to have
 $h_r> h_s$ in $(\beta,\beta+\epsilon')$. 
Therefore, $\beta=s$ and $h_r(u)\le h_s(u),
\forall u\in [0,s)$.\end{proof}

\begin{rmk}
Lemma 2.3 immediately implies that there is a unique solution for the equation (2.2).
\end{rmk}
We would like to explain briefly  the condition $\lim_{u\to r}h_r(u)=\infty$ has implied
the corresponding metric is complete on $T^rM.$ 

The distance from  center to the boundary is (c.f. p.157. \cite {S})
\begin{equation}L=\frac {1}{\sqrt 2}\int_0^r \sqrt {h_r''(u)}\ du.\end{equation}
For $L$ to be infinity, it is sufficient to show 
\begin{equation}\sqrt {h_r''(u)}> \frac {1}{r-u}\end{equation}  when
$u\to r$, i.e., we need a comparison of $\sqrt
{h_r''(u)}=\left (e^{
\lambda h_r(u)}\frac {\mathcal D_{M}(u)}{(h_r'(u))^{n-1}}\right )^{\frac 12}$
with $\frac {1}{r-u}$. If $\frac {\mathcal D_{
M}(u)}{(h_r'(u))^{n-1}}> 0$ as $u\to r$, the distance $L=\infty$ since $\sqrt
{h_r''(u)}$  then grows  exponentially. 

The worst case is  $\frac {\mathcal D_{
M}(u)}{(h_r'(u))^{n-1}}\to 0$ as $u\to r$. After a translation, we may set the origin at $\{u=r\}$ and
call the new coordinate by $x$. Let the order of vanishing at $x=0$ of $\left (\frac {\mathcal
D_{ M}(x)}{(h_r'(x))^{n-1}}\right )^{\frac 12}\simeq x^k, k\in\mathbb N$; the order of going infinity
of $h_r(x)$ is $x^{-\alpha}, \alpha>0$. Repeatedly applying the L'H$\hat o$pital's rule
shows, for $\alpha>0$, $$\lim_{x\to 0} x^k e^{x^{-\alpha}}=\infty.$$
Thus (2.11) holds and $L=\infty$.

\

\section{Complete Ricci-flat metric in $TM$ through a rescaling process}
\setcounter{equation}{0}
\setcounter{thm}{0}

In this section, we fix $M$ to be  $\mathbb R^n$ or  a compact symmetric space of rank-one,
{\it i.e.}, it is either $\mathbb R^n$ or one of  the round sphere, the real projective space, the
complex projective space, the quaternionic projective space  or  the Caley plane. 

The adapted complex structure is defined on the whole tangent bundle $TM$.  For any $r>0$ and 
any $\lambda_r>0$, there exists a unique real-analytic function $f_r(u),
u=\sqrt
\rho,$ satisfying the ODE
\begin{equation}\left\{\begin{aligned}
f_r''(u)(f_r'(u))^{n-1}&=e^{ \lambda_r f_r(u)}\mathcal D_{ M}(u),\;
u\in [0,r);\\
\lim_{u\to r}f_r(u)&=\infty.\end{aligned}\right.\end{equation}
Indeed, $f_r$ is a \K potential of the complete \ke \ of Ricci curvature $-\lambda_r$ in 
\ $T^rM$. Since 
 the whole tangent bundle $TM$ is exhausted by \gt s $\{T^rM\}_{ 0<r<\infty},$ a $C^2$-convergence
of the family $\{f_r\}_{r>0}$ will lead to the existence of a \K metric with Ricci curvature
$-\lim_{r\to\infty}\lambda_r$.

 Let $\{\lambda_r\}_{r>0}$ be a  decreasing sequence of positive numbers, 
$\lambda_r<\lambda_s$ whenever $r>s$,  such that $\lim_{r\to \infty}\lambda_r=0.$

A comparison analogous to Lemma 2.3 has played an essential role in the convergent argument. In this
rescaling setting, we are looking for a comparison between $f_r$ and $f_s$ where $f_r$ satisfies
(3.1) and $f_s$ is the unique solution of the following:
\begin{equation}\left\{\begin{aligned}
f_s''(u)(f_s'(u))^{n-1}&=e^{ \lambda_s f_s(u)}\mathcal D_{ M}(u),\;
u\in [0,s);\\
\lim_{u\to s}f_s(u)&=\infty.\end{aligned}\right.\end{equation}

In each $T^rM$, we fix the Ricci curvature to be $-1$ and let $h_r$ be the unique
solution of $(2.2)$ in $T^rM$ with $\lambda=1$:
\begin{equation}\left\{\begin{aligned}h_r''(u)(h_r'(u))^{n-1}&=e^{ h_r(u)}\mathcal D_{
M}(u),\; u\in [0,r);\\
\lim_{u\to r}h_r(u)&=\infty.\end{aligned}\right.\end{equation} 
Each $h_r(u)$ is an increasing function in $u$. We denote it's minimum as $a_r$,
\begin{equation} a_r:=h_r(0)=\inf_{u\in [0,r)}h_r(u).
\end {equation} 
\begin{lem}
$\{a_r\}_{r>0}$ is a decreasing sequence  and $\lim_{r\to\infty}a_r=-\infty$. 
\end{lem}
\begin{proof}
By Lemma 2.3, the sequence $\{a_r\}_{r>0}$ is decreasing.

Suppose $\lim_{r\to\infty}a_r=c$ for some
real number $c$. Given $\epsilon>0$ there exists N such that $|a_r-c|<\epsilon$ whenever
$r\ge N$. In other words, 
\begin{equation}|a_l-a_N|<2\epsilon \text { for any } l>N.
\end{equation}
Observing from the equations
\begin{equation}\begin{aligned} 
h_l'(t)&=\left (\int_0^{t} ne^{  h_l(u)}\mathcal{D} (u)\right )^{\frac 1n},\; t\in (0,l);\\
h_N'(t)&=\left (\int_0^{t} ne^{  h_N(u)}\mathcal{D} (u)\right )^{\frac 1n},\; t\in (0,N),
\end{aligned}\end{equation}
the difference of $h_l$ and $h_N$ is decided by their initial values $a_l$ and $a_N$. Since
$|a_l-a_N|<2\epsilon$  for any $l>N$ and $\lim_{u\to N}h_N(u)=\infty$, there exists a positive number
$\alpha$ such that $h_l(N+\alpha)=\infty$. This is not possible if $l$ is taken to be sufficiently
large. $\lim_{r\to\infty}a_r=-\infty$ is concluded.
\end{proof}

Define the function 
\begin{equation}H_r(u)=h_r(u)-a_r.\end{equation}
For any given $r>0$,  $H_r(0)=0$ and $H_r(u)\ge 0, \forall u\in [0,r)$ and
\begin{equation}\left (e^{\frac {-a_r}{n}}H_r\right)''\left (e^{\frac {-a_r}{n}}H'_r\right)^{n-1}=
\exp {\left (e^{\frac {a_r}{n}} e^{\frac {-a_r}{n}}H_r\right)}\mathcal D_{ M}.\end{equation}
The function \begin{equation}K_r:= e^{\frac {-a_r}{n}}H_r\end{equation} satisfies
\begin{equation}K_r''(u)(K_r'(u))^{n-1}=\exp(e^{\frac {a_r }{n}}{K_r(u)})\mathcal D_{
M}(u),\;
\forall u\in [0,r)\end{equation} 
and $\lim_{u\to r}K_r(u)=\infty.$ By the uniqueness of the solution, $K_r$
is a \K potential of the complete  \ke\ of Ricci
curvature 
$-e^{\frac {a_r}{n}}$ in $T^rM$.
By (3.10), 
\begin{equation}\begin{aligned}
K_r''(u)(K_r'(u))^{n-1}&=\exp(e^{\frac {a_r }{n}}{K_r(u)})\mathcal D_{M}(u)\\
&=\exp (h_r -a_r)\mathcal D_{M}(u).
\end{aligned}\end{equation} 
Then 
\begin{equation}
(K_r'(u))^{n}=n\int \exp (h_r -a_r)\mathcal D_{M}(u).
\end{equation}
If we can show $h_r-a_r< h_s-a_s$, then we have 
\begin{equation}
K_r'(u)<K_r'(u).
\end{equation}
We also have
\begin{equation}
(h_r'(u))^{n}=n\int \exp h_r\mathcal D_{M}(u).
\end{equation}
Since $h_r< h_s$, then $h'_r< h'_s.$
Now 
\begin{equation}\begin{aligned}
h_r-a_r&=\int_0^1 h_r' (tu) udt\\
&<\int_0^1 h_s' (tu) udt\\
&= h_s-a_s.
\end{aligned}\end{equation} 
 $K_r(0)=K_s(0)$, (3.13) will imply that 
\begin{equation}
K_r(u)\le K_s(u),\;\forall u\in [0,s)
\end{equation}
which has provided a uniform upper bound for the family $\{K_r\}_{r>0}$ in compact subsets. A uniform
lower bound is easily obtained since for any $r>0$ and any $u\in [0,r)$, 
$$K_r(u)=e^{\frac {-a_r}{n}}H_r(u)\ge 0.$$
The following proposition has been concluded.
\begin{prop}
 $\{K_r\}_{r>0}$ has converged uniformly on any compact subset of $[0,\infty)$ to a continuous
function $K$.
$$\lim_{r\to \infty}K_r(u):=K(u),\; u\in [0,\infty).$$
\end{prop}
The goal is to show this function $K$ is a \K potential of a Ricci-flat metric in $TM$. For
this to work,  it is sufficient to find some uniform bounds for the first derivatives and the
second derivatives of the family $\{K_r\}_{r>0}$ in compact subsets. 
For a fixed compact set $A\subset [0,\infty)$, the restriction of $\mathcal D_{
M}$ in $A$ is bounded. By (3.16), for any $r>s$,
\begin{equation}\begin{aligned}
e^{\frac {a_r }{n}}{K_r(u)}&\le e^{\frac {a_r }{n}}{K_s(u)},\; \forall
u\in [0,s);\\
&\le K_s(u),\; \forall
u\in [0,s),\; r\gg 1,
\end{aligned}\end{equation}
where the last inequality comes from the fact that  $\lim_{r\to\infty}a_r=-\infty$. On the other
hand, $e^{\frac {a_r }{n}}{K_r(u)}=H_r(u)\ge 0.$
 Therefore,
 there exists  $c>0$ such that  
\begin{equation}
0\le \mathcal D_{ M}(u)   \le\exp(e^{\frac {a_r }{n}}{K_r(u)})\mathcal D_{ M}(u)\le
c,\; u\in A,\; r\gg 1.
\end{equation}
\begin{lem} In any compact set $A\subset [0,\delta]\subset [0,\infty)$, the family $\{K_r'\}$ is
uniformly bounded:
$$ 0\le K_r'(t)\le \root n\of {nc}\; \root n\of\delta,\;
t\in [0,\delta],\; \forall r\gg 1.$$
\end{lem}
\begin{proof}
The fact that $K'_r(0)=0$ along with (3.10)  and
(3.18) then  implies for $t\in [0,\delta],\; r\gg 1$,
\begin{equation}\begin{aligned}
(K'_r(t))^{n}&=\int_0^{t}n\exp(e^{\frac {a_r }{n}}{K_r(u)})\mathcal D_{ M}(u)du\le ntc;\\
(K'_r(t))^{n}&=\int_0^{t}n\exp(e^{\frac {a_r }{n}}{K_r(u)})\mathcal D_{ M}(u)du\ge
\int_0^{t}n\mathcal D_{ M}(u)du
\end{aligned}\end{equation}
for some $c>0$. Thus,
\begin{equation} 0\le \left (\int_0^{t}n\mathcal D_{ M}(u)du\right )^{\frac 1n}    \le
K_r'(t)\le
\root n\of {nc}\;
\root n\of\delta,\; t\in [0,\delta],\; \forall r\gg 1.
\end {equation}
This uniform estimate in $ [0,\delta]$ of course has implied a uniform estimate in the compact
set $A$.
\end{proof}
The next step is to find some uniform bound on the second derivatives.
\begin{lem}
Given a compact set $A\subset [0,\delta]\subset [0,\infty)$, there exists a constant $C>0$ such that 
$$-C\le K''_r (u)\le C,\; u\in A, \;\forall  r\gg 1.$$
\end{lem}
\begin{proof}
 A uniform lower bound for $K''_r(u)$ is already available since 
\begin{equation}
K''_r(u)=e^{\frac {-a_r}{n}}H''_r(u)=e^{\frac {-a_r}{n}}h''_r(u)>0, \forall u\in [0,r),
r\in (0,\infty).
\end{equation} For a uniform upper bound, we consider two kinds of compact sets: $
A=[\epsilon,\delta], \epsilon >0$, and 
$ A=[0,\delta].$
For the first case, (3.20) shows there exists a constant $d>0$  such
that $$d\le K_r'(\epsilon),\;\forall r\gg 1.$$
By (3.21), $K_r'$ is an increasing function in $u$, so
 $$
K_r'(u)\ge K_r'(\epsilon) \ge d\text { for any } u\ge\epsilon.
$$ 
(3.10) (3.18) and (3.21) have concluded a uniform upper bound: 
\begin{equation}0\le K_r''(u)\le \frac { c}{(K_r'(u))^{n-1}}\le \frac { c}{d^{n-1}},\; u\in
[\epsilon,\delta],\;\forall  r\gg 1.\end{equation}

For the second case, it  amounts to show  $K''_r(0)$ have a uniform upper bound for
 $r\gg 1$. By (3.21),
\begin{equation}
K''_r(0)=e^{\frac {-a_r}{n}}h''_r(0)=e^{\frac {-a_r}{n}}e^{\frac {h_r(0)}{n}}=1, \forall \;
r\in (0,\infty),
\end{equation}        
where the second equality is obtained from Lemma 2.2.
The lemma is therefore concluded.
 \end{proof}
\begin{thm}
Let $M$ be a compact rank-one symmetric space or $M=\Bbb R^n$. There exists a complete
 Ricci-flat metric in $TM$  obtained from an exhaustion by  \ke s in
$\{T^rM\}_{0<r<\infty}$.
\end{thm}
\begin{proof}
Let the function $K$ in $TM$ be the one defined in Proposition 3.2.
Then lemmas 3.3 and 3.4 have asserted the uniform convergence of first derivatives  
and second derivatives in any compact subset of $[0,\infty)$. Therefore,
$$K'(u)=\lim_{r\to\infty}K'_r(u);\; K''(u)=\lim_{r\to\infty}K''_r(u),\; u\in
[0,\infty).$$
Since $K_r$
is a \K potential of a  \ke\ of Ricci
curvature 
$-e^{\frac {a_r}{n}}$ in $T^rM$, the function $K$ is a \K potential of some \ke\ in $TM$ 
with Ricci curvature $-\lim_{r\to \infty}e^{\frac {a_r}{n}}$.
 By Lemma 3.1, $\lim_{r\to\infty}a_r\to
-\infty$, therefore
$K$ is a \K potential of a Ricci-flat metric in $TM$.
Furthermore,
\begin{equation}K''(u)(K'(u))^{n-1}= \mathcal D_{ M}(u),\;
\forall u\in [0,\infty),\end{equation}
which implies
\begin{equation}
(K'(t))^n=\int_0^{t} n\mathcal D_{ M}(u)du,\;\forall t\in (0,\infty)
\end{equation}
Examining the list  after (2.2), it is clear that $\mathcal D_{ M}$ grows exponentially when $M$ is
a compact rank-one symmetric space. Hence $K'(t)$ grows exponentially as well.  On the other hand,
\begin{equation}
K'(t)=\int_0^{t} K''(t)du,\;\forall t\in (0,\infty)
\end{equation}
so $K''(t)$ goes to $\infty$ in an exponential way as $t\to \infty.$

When $M=\mathbb R^n, (K'(t))^n=\int_0^{t} nu^{n-1}du=t^n$ and $K''(t)=1, \;\forall t\in (0,\infty).$

 Since the distance from the center to the boundary is given by $\frac {1}{\sqrt 2}\int_0^{\infty}
\sqrt {K''(u)}\ du$ which equals to $\infty$ for all the above mentioned cases.
The  metrics are complete.\end{proof}

\

 The resulting \K potential in $T\mathbb R^n$ is a solution of 
$h''(u)(h'(u))^{n-1}=u^{n-1}$ which has
taken the form $h(u)=\frac {u^2}2+c$ for any constant $c$.
It is known in the Euclidean case that $u(z)=|y|$ and the \K
potential
$h(z)=\frac {|y|^2}2+c=\frac {-1}8\sum_{j}(z_{\bar j}^2+z_{j}^2-2z_j z_{\bar j}).$ The
corresponding Ricci-flat metric is simply  the standard Euclidean metric in $\mathbb
C^n$.

It is interesting to observe that from two different exhaustions, one through the balls discussed
in $\S 2$, and the other one through $T^r\mathbb R^n$ studied in this section, the resulting
Ricci-flat metrics are the same.

 Stenzel \cite {S} has also worked on the tangent bundle of compact rank-one symmetric spaces. By
the transitivity of the rank-one symmetry,  the defining \ma
 equation for a Ricci-flat metric could be reduced to an ordinary differential equation. Working
directly on the solvability and the completeness of this  ordinary differential equation,
Stenzel was able
to show the existence  a complete Ricci-flat metric in $TM$ for compact rank-one
$M$.

Although our resulting Ricci-flat metric  through the rescaling process turns out to be the same
with the one constructed by Stenzel in $TM$.  We think  the
approach here is interesting and we are looking for some further investigation.

\section{ \ke\ of negative Ricci in $T^{\pi}H$}
\setcounter{equation}{0}
\setcounter{thm}{0}

Using an exhaustion by  smooth bounded \spc domains, Cheng-Yau were able to confirm the
existence of a  \ke\ with negative Ricci curvature in any bounded pseudoconvex domain. With
very mild regularity condition, the metric could be proved to be complete. Later on, Mok-Yau
have extended the existence and completeness to any bounded Stein domain by showing the
exhaustion has a uniform convergent limit.
The convergence of the exhaustion has strongly relied on the boundedness of the domain $D$.
An essential point is
$D\Subset B(0,R)$ for some large $R$, so that the \ke s in  the exhaustion family could be made
comparison with 
 the Poincar\'e metric of  $B(0,R)$ to get hold  of a uniform lower bound. 

It is not clear whether such kind of metric exists in unbounded pseudoconvex domain in $\Bbb
C^n$ or not. In this case, the comparison theorem of [C-Y] has provided a uniform upper
bound for the sequence. However, it is no longer clear whether there exists any uniform lower bound 
 or not.
\

In this section, we consider \gt s over  the real-hyperbolic space $H^n$, a non-compact rank-one
symmetric space.
 Since
$H^n$ is co-compact, there exists a discrete subgroup $G\subset Isom (H^n)$ so that $\hat H=H/G$ is
a compact real-analytic Riemannian manifold. It was shown in \cite {K1}\cite {K2} that  the maximal
radius for $\hat H^n$ is $\pi.$

For any $r<\pi$, $T^r\hat H$ is a bounded
Stein domain with  smooth
\spc boundary in $T^{\pi}\hat H$ and the existence of a complete \ke\ with negative Ricci is
guaranteed.
 Exhaustion  by such an increasing family of Stein domains, the
manifold $T^{\pi}\hat H$ itself is then a Stein manifold.

By the nature of the adapted complex structure, the universal covering 
 $T^rH$  of $T^r\hat H$ has shared the same maximal radius, i.e., $r_{max} (H)=\pi$ and 
the existence of a complete \ke\ with negative Ricci in $T^rH, \; r<\pi$, is also guaranteed. 
Furthermore, $T^{\pi}H$ is a Stein manifold since it is the universal covering of the Stein manifold
$T^{\pi}\hat H$.

 It is not clear at all whether this $T^{\pi}H$ could sit inside any K\"ahler
manifold as a bounded domain or not. Thus, we can't apply the main theorem in \cite {M-Y} to
conclude the existence of a  \ke\ with negative Ricci  in it. The goal of this
section is to show the existence of such a metric in $T^{\pi}H$. 

$\mathcal{D}_{H^n}(u)=(\sin(u))^{n-1}$  and $r_{max}(H)=\pi$. Therefore,
\gt\
$T^rH^n$ has  existed for any $r\in (0,\pi),$
and a \K potential for the complete
\ke\ of Ricci $-1$ in $T^rM$ is given by
\begin{equation}\left\{\begin{aligned}
h_r''(u)(h_r'(u))^{n-1}&=e^{  h_r(u)}(\sin (u))^{n-1},\; u\in [0,r);\\
\lim_{u\to r}h_r(u)&=\infty.
\end{aligned}\right.\end{equation}
 Furthermore,  $h_r$ is  real-valued and real-analytic.

The K\"ahler manifold $T^{\pi}H$ is exhausted by   the family $\{T^{r}H: 0<r<\pi\}$ in the sense that
$T^{r}H$ is an increasing family of \gt s and  $T^{\pi}H=\cup_{0<r<\pi}T^{r}H$.  $T^{r}H$
is not relatively compact in  $T^{\pi}H$ since $H$ is not compact. 

Let $a_r\in [0,r)$ be the largest number such that $h_r(a_r)=0$. If there is no such
$a_r$, then $h_r(u)>0$ for all $u\in [0,r)$.
\begin{lem}
$h_r(0)\ge -\pi \root n\of n \root n\of { \pi},\; \forall r\in (0,\pi).$
\end{lem}
\begin{proof} Given $r\in (0,\pi)$, if $h_r(u)>0$ for all $u\in [0,r)$ the statement automatically
holds. Without loss of generality, we may assume  $a_r\ge 0$.

Since $h_r$ is a monotonically  increasing function of $u$, $h_r(u)\le
0,\;\forall u\in [0, a_r].$ For any given $t\in [0,a_r]$, the equation  (4.1) implies
\begin{equation}
(h_r'(t))^n=\int _0^t ne^{ h_r(u)}(\sin (u))^{n-1}
\le nt.
\end{equation}
Since $h_r'(u)$ is increasing in $u$,
\begin{equation}\begin{aligned}
-h_r(0)=h_r(a_r)-h_r(0)&=\int_0^{a_r}h'_r(t)dt\\
&\le \int_0^{a_r}h'_r(a_r)dt\\
&\le a_r \root n\of n \root n\of {a_r}\\
&\le \pi \root n\of n \root n\of \pi.
\end{aligned}\end{equation}
The lemma is concluded.\end{proof}

As an increasing function of $u$,
Lemma 4.1 has provided  a uniform lower bound for $h_r(u)$.
That is, 
\begin{equation} inf_{u\in [0,r)}h_r(u)\ge h_r(0)\ge -\pi \root n\of n \root n\of {\pi},\;\forall
r\in (0,\pi).\end{equation}

Summing up Lemma 2.3 and Lemma 4.1, it is clear that for given compact set $K\subset
[0,\delta]\subset [0,\pi)$, there exists a $c>0$ such that 
\begin{equation}-c\le h_r (u)\le c,\;\;\forall u\in K,\;\forall r > \delta\end{equation}
and the  limit exists
\begin{equation}h(u):=\lim_{r\to\pi}h_r(u),\; u\in [0,\pi).\end{equation}

The goal is to show this function $h$ is a \K potential of a \ke\ in $T^{\pi}H$. For this
to work,  it is sufficient to find some uniform bounds for the first derivatives and the
second derivatives of the family $\{h_r\}$ in compact subsets. 
\begin{lem} In any compact set $K\subset [0,\delta]\subset [0,r)$, the family $\{h_r'\}_{r>\delta}$
is uniformly bounded: there exists a $c>0$ such that
$$ 0\le h_r'(t)\le \root n\of n\; \root n\of\delta\; e^{\frac {c}{n}},\;\forall\; t\in
K,\; r>\delta.$$
\end{lem}
\begin{proof}
A constant $c$ could be chosen from (4.5), then for  $r >\delta$, the following
inequality holds in $ K$: 
\begin{equation}n e^{-c}(\sin(u))^{n-1}\le ne^{ h_r(u)}(\sin(u))^{n-1}\le n e^{c}.\end {equation} 
Since $h'_r(0)=0$, (4.2) then  implies
\begin{equation} n\; e^{-c}\int _0^t (\sin u)^{n-1}\le (h_r'(t))^n\le n\; t\; e^{
c},\;\forall t\in [0,\delta],\; r>\delta.\end {equation}Thus
\begin{equation} 0\le h_r'(t)\le \root n\of n\; \root n\of\delta\; e^{\frac {c}{n}},\;\forall
t\in [0,\delta],\; r>\delta.
\end {equation}
This uniform estimate in $ [0,\delta]$ of course has implied a uniform estimate in $K$.
\end{proof}

The next step is to derive a uniform bound on the second derivatives.
\begin{lem}
Given compact set $K\subset [0,\delta]\subset [0,r)$, there exist a constant $C>0$ such that 
$$-C\le h''_r (u)\le C,\;\forall u\in K, \; r>\delta.$$
\end{lem}
\begin{proof}
 Lemma 2.2 has provided a uniform lower bound $h''_r(u)>0, \forall u\in [0,r), r\in
(0,\pi).$ For a uniform upper bound, we consider two kinds of compact sets: $
K=[\epsilon,\delta], \epsilon >0$ and 
$ K=[0,\delta].$

For the first case, (4.8) shows there exists a constant $d>0$  such
that $$d\le h_r'(\epsilon),\;\forall r>\delta.$$
As $h''\ge 0$, $h_r'(u)\ge h_r'(\epsilon)\ge d$ for any $u>\epsilon$. By (4.1) and (4.9), 
\begin{equation}0\le h_r''(u)\le \frac { e^{ c}}{(h_r'(u))^{n-1}}\le \frac { e^{
c}}{d^{n-1}},\;  u\in
K, \forall r>\delta.\end{equation}

For the second case, it amounts to show  $h''_r(0)$ have a uniform upper bound for
 $r>\delta$. By Lemma 2.2 and (4.5) there exists a constant L such that 
\begin{equation} h''_r(0)=e^{\frac {h_r(0)}{n}}\le L,\;\forall r>\delta.\end{equation}
(4.10) along with (4.11) has proved the lemma.
\end{proof}
\begin{thm}
There exists a complete \ke\ of Ricci curvature $-1$ in $T^{\pi}H$.
\end{thm}
\begin{proof}Lemmas 4.2 and 4.3
have concluded that   the two  families $\{h_r'\}$ and
$\{h_r''\}$ have converged uniformly  in any compact subset of $[0,\pi)$. This shows the
convergence
$h_r(u)\to h(u)$ is good up to  second orders. The function $h$ is then a \K potential of a
\ke\ of Ricci curvature $-1$ in $T^{\pi}H$ and $\lim_{u\to\pi}h(u)=\infty.$ Furthermore,
\begin{equation}h''(u)(h'(u))^{n-1}=e^{ h(u)}(\sin (u))^{n-1},\;\forall u\in [0,\pi).
\end{equation}
Although $\lim_{u\to\pi}\sin(u)=0$, the rapid exponential growth of $e^{ h(u)}$  still guarantees
that
$\lim_{u\to\pi}e^{ h(u)}(\sin (u))^{n-1}=\infty$ in the exponential way.
The equation  
\begin{equation}
(h'(t))^n=\int_0^{t} ne^{ h(u)}(\sin (u))^{n-1}du,\;\forall t\in (0,\pi)
\end{equation}
has shown that the $h'$ is increasing to $\infty$ in an exponential way near $\pi$. The fact that
the second derivative
$h''$ has increased to $\infty$ exponentially near $\pi$ is obtained through the following  equation
$$h'(t)=\int_{0}^{t} h''(u)du.$$ 
We thus conclude the  metric is complete.
\end{proof}

\

 After this work has been done, we found   the authors in \cite {B-H-H} have discovered that
maximal
\gt s over any rank-one space is Hermitian symmetric. And it is well-known, c.f. \cite {H}, 
 that any Hermitian symmetric space is biholomorphic to a
bounded domain in $\Bbb C^n$. 
$T^{\pi}H^n$ is then a bounded domain of holomorphy in
$\Bbb C^n$. 
 The existence of a complete \ke\ with negative Ricci curvature in $T^{\pi}H^n$ could  be 
concluded from \cite {M-Y} directly.

At the end of \cite {K2}, we have shown near the center, the holomorphic sectional curvatures
along  the Monge-Amp\`ere are negative when the center of the \gt s is of compact rank-one or is the
Euclidean space. We were not able to reach any definite result for the real-hyperbolic space. With
the machineries developed in this section, we conclude:
\begin{prop}Let $k_r$ denote the complete \ke\ of Ricci curvature $-(n+1)$ in
the
\gt\ $T^rH$. There exists an $\epsilon>0$ such that for any $r\in (0, 
\frac{\pi}2+\epsilon)$ holomorphic sectional curvatures of
$k_r$ along  the Monge-Amp\`ere leaves of  $T^rH$ are negative near the center $H$.
\end{prop}
\begin{proof}
It amounts to show  $b=\frac 12\exp {\frac{n+1}{n}a}>\frac {n-1}{6(n+1)}$ where
$a=h_r(0)$. The \gt\ $T^{\frac{\pi}2}H$ is biholomorphic to the ball and an explicit
solution
$h_{\frac{\pi}2}(u)=-\log \cos u$
 to the defining ODE. Since $h_{\frac{\pi}2}(0)=0$, we conclude $h_{r}(0)\ge 0$ for any
$r\in (0, \frac{\pi}2)$. By the continuity, there exists an $\epsilon>0$ such that 
$h_{\frac{\pi}2+\epsilon}(0)$ is quite close to $0$. In this cases, $b\simeq \frac
12>\frac {n-1}{6(n+1)}$.
\end{proof}

\

Institute of Mathematics, Academia Sinica, Taipei 11529, Taiwan

{\it E-mail address: kan@math.sinica.edu.tw}

\end{document}